\documentclass[11pt]{amsart}

\newcommand{\End}{\operatorname{End}}
\newcommand{\GL}{\operatorname{GL}}
\usepackage{fullpage}
\usepackage{amssymb}
\usepackage{amsmath}
\usepackage{amsxtra}
\usepackage{amscd}
\usepackage{graphicx}
\usepackage{amsfonts}
\usepackage{pb-diagram}
\usepackage{float}
\usepackage{enumerate}
\usepackage{dsfont}
\usepackage{multirow}
\usepackage{color}
\usepackage{rotating}
\usepackage{url}
\usepackage{enumitem}
\usepackage[bookmarks=false]{hyperref}
\definecolor{darkgreen}{rgb}{0,0.4,0.1}
\definecolor{darkpurple}{rgb}{0.7,0.0,0.4}
\hypersetup{
	colorlinks=true,       	
	linkcolor=darkpurple,     
	citecolor=darkgreen,    
	filecolor=magenta,      
	urlcolor=blue			
}

\DeclareMathOperator{\rank}{rank}

\numberwithin{equation}{section}
\newtheorem{thm}{Theorem}[section]
\newtheorem{lem}[thm]{Lemma}
\newtheorem{cor}[thm]{Corollary}
\newtheorem{prop}[thm]{Proposition}

\newtheorem{defn}[thm]{Definition}
\newtheorem{conj}[thm]{Conjecture}

\newtheorem{rem}[thm]{Remark}

\usepackage{geometry}

\usepackage{amsthm}
\usepackage{amssymb}
\usepackage{bbm}
\usepackage{cancel}
\usepackage{tikz}
\usetikzlibrary{shapes}

\makeatletter
\numberwithin{equation}{section}
\numberwithin{figure}{section}
\numberwithin{table}{section}
\theoremstyle{plain}

\makeatother
\renewcommand\part[1]{\textbf{(#1)}}

\definecolor{mediterranean}{cmyk}{.67,0,.08,.3}
\definecolor{rose}{cmyk}{0,1.00,.20,0}
\definecolor{darkorchid}{cmyk}{.6,.9,0,.05}
\definecolor{butterfly}{cmyk}{.95,.59,0,.10}
\definecolor{springgreen}{cmyk}{1.00,0,.70,.02}
\definecolor{darkred}{cmyk}{0,1,1,.3}
\definecolor{nectarine}{cmyk}{0,0.70,1.00,0}
\definecolor{icyblue}{cmyk}{.84,.25,0,.06}

\usetikzlibrary{decorations.markings}
\usetikzlibrary{arrows,shapes,positioning}
\tikzset{every loop/.style={min distance=10mm,looseness=10}}

\newcommand\undermat[2]{
	\makebox[0pt][l]{$\smash{\underbrace{\phantom{%
					\begin{matrix}#2\end{matrix}}}_{\text{$#1$}}}$}#2}

\title{The center of the walled Brauer algebra $B_{r,1}(\delta)$}

\author{Eirini Chavli}
\address{Institut für Diskrete Strukturen und Symbolisches Rechnen, Universit\"at Stuttgart, Pfaffenwaldring 57, 
	70569 Stuttgart, Germany.}
\email{eirini.chavli@mathematik.uni-stuttgart.de}

\author{Maud De Visscher}
\address{Department of Mathematics, City University of London, Northampton Square, London EC1V 0HB, United Kingdom.}
\email{maud.devisscher.1@city.ac.uk}

\author{Alison Parker}
\address{School of Mathematics, University of Leeds, LS2 9JT, United Kingdom.}
\email{a.e.parker@leeds.ac.uk}

\author{Sarah Salmon}
\address{Department of Mathematics, University of Colorado Boulder, Campus Box 395, 80309-0395, United States.}
\email{Sarah.Salmon@colorado.edu}

\author{Ulrica Wilson} 
\address{Science, Technology, Engineering and Mathematics Division Faculty, Morehouse College,830 Westview Drive, S.W.
	Atlanta, GA 30314, United States.}
\email{ulrica.wilson@morehouse.edu}

\thanks{We would like to thank the Women in Noncommutative
	Algebra and Representation Theory (WINART2) workshop, held at the University
	of Leeds in May 2019, where this project started.}
\begin{document}
			\maketitle
			\begin{center}
			\emph{To the memory of Emmelia Kokota}
			\end{center} 
		\begin{abstract}We show that the centre of the walled Brauer algebra $B_{r,1}(\delta)$ over the complex field $\mathbb{C}$, for any parameter $\delta\in \mathbb{C}$, is generated by the supersymmetric polynomials evaluated at the Jucys-Murphy elements. Moreover, 	we prove that its dimension is independent of the parameter $\delta$.	\end{abstract}
		
	\section{Introduction} Let $V=\mathbb{C}^n$ be the natural representation of $\GL_n(\mathbb{C})$ and let $r$ be a positive integer. There are two actions on the $r$th tensor product $V^{\otimes r}$: the first action is by $\GL_n(\mathbb{C})$, which is the (left) diagonal action. The second one  is the (right) action of the symmetric group $S_r$ which permutes the tensor factors. One can easily see that these two actions compute. However, there is a stronger relation between them, as asserted by \emph{Schur-Weyl duality} \cite{W}, namely that the span of the image of $S_r$ and $\GL_n(\mathbb{C})$ in $\End(V^{\otimes r})$ are centralizers of each other.
	
	The Brauer algebra $B_r(\delta)$ was introduced  by Brauer \cite{Br}, in the context of classical invariant theory, to play the role of the
	symmetric group in a corresponding Schur-Weyl duality for the
	orthogonal  groups (for $\delta$ positive integer) and for the  symplectic groups (for $\delta$ negative integer). 
	
	
	The walled Brauer algebra $B_{r,s}(\delta)$ is a subalgebra of $B_{r+s}(\delta)$ and it was introduced independently by Turaev  and by Koike \cite{K,T}. It appears in another generalisation of Schur-Weyl duality, when considering the action of  $\GL_n(\mathbb{C})$ on mixed tensor space $V^{\otimes r}\otimes (V^*)^{\otimes s}$.
	
Brauer algebras and walled Brauer algebras are particular examples of \emph{diagram} algebras. Their representation theory has been studied extensively, see for example \cite{CDM,CDDM,CD,Mar,Rui} and references therein. In particular, they are semisimple unless the parameter $\delta$ is a small integer compared to $r$ (and $s$). In the semisimple case, there is an explicit construction of the simple modules. Even in the non-semisimple cases, their representation theory is well-understood: they are cellular algebras and we have an explicit combinatorial formula for their decomposition numbers.  One question however remains open in the non-semisimple case, namely the description of the centre of these algebras.  The aim of this paper is to make some progress in the case of the walled Brauer algebra. 

Much of the work to study the structure of these diagram algebras has been done by analogy with the group algebra of the symmetric group. The study of their centres is no exception. So we start by recalling the classical results on the centre of $\mathbb{C}S_r$.

As for the centre of any group algebras of a finite group, the centre of $\mathbb{C}S_r$ has a basis given by the conjugacy class sums, indexed in this case by partitions of $r$. But there is another description, in terms of the so-called Jucys-Murphy elements, first introduced in \cite{J,Mu}, defined by
$$L_k:=\sum_{j=1}^{k-1}(j,k),\,\, 1\leq k\leq r.$$
These elements play an essential role in the Okounkov-Vershik approach to the representation theory of the symmetric groups \cite{OV}. The center of $\mathbb{C}S_r$ can be described as the algebra of all the symmetric polynomials evaluated in the Jucys-Murphy elements (for more details, see \cite{J,M}).

Analogues of these Jucys-Murphy elements for the algebra $B_{r,s}(\delta)$ were introduced in \cite{BS, RS, SS} and conjectures were made about using some versions of symmetric polynomials in these to describe the centre of the walled Brauer algebra (see \cite[Remark 2.6]{BS} and \cite[Conjecture 7.3]{SS}).
In \cite{JK} Jung and Kim, inspired by \cite[Conjecture 7.3]{SS},  introduced a renormalisation of these  Jucys-Murphy elements, which we denote again as $L_1,\dots, L_{r+s}$. They proved that \emph{supersymmetric polynomials} in these elements are central in $B_{r,s}(\delta)$ for all values of the parameter $\delta$ and that moreover,  when $B_{r,s}(\delta)$ is semisimple, this is in fact the whole centre. The key fact here is that when an algebra is semisimple, the dimension of its centre is given by the number of isomorphism classes of simple modules. 
Jung and Kim conjectured in \cite[Conjecture 5.4]{JK} that the result still hold in all cases. The aim of this paper is to investigate this question. We give a strategy for the general case, following the work of Shalile \cite{Sha} and proves this conjecture in the case when $r$ is arbitrary and $s=1$.

\medskip

{\bf Structure of the paper.}
Section 2 contains the definition of the walled Brauer algebra in terms of its diagram basis and multiplication. One should note that $B_{r,s}(\delta)$ contains $\mathbb{C}(S_r\times S_s)$ as a subalgebra and that in fact, it is generated by this subalgebra together with one more element, which we denote by $e$. We also recall the definition of the Jucys-Murphy elements, supersymmetric polynomials and state Jung-Kim's result for the centre in the semisimple case. In Section 3, we obtain a lower bound for the dimension of the centre in the general case, using Jung-Kim's result and a base-change argument. 
In Section 4 we recall the construction of the generalised walled cycle types due to Shalile \cite{Sha}. The sums of all diagrams with the same cycle type give a basis  for the centraliser algebra of the product of symmetric group $S_r\times S_s$ in the walled Brauer algebra $B_{r,s}(\delta)$.  Thus, in order to find the centre of the walled Brauer algebra, we only need to impose one additional condition, namely commutation with the extra generator $e$. This,  together with the lower bound obtained in Section 4 gives a general strategy to find the centre. We then apply this strategy explicitly in Section 5 for the special case $B_{r,1}(\delta)$.

	\section{Preliminaries} 
	\subsection{Walled Brauer algebras}\label{def} Let $r$ and $s$ be nonnegative integers. An $(r,s)$-\emph{walled Brauer diagram} is a graph drawn in a rectangle with $(r+s)$ vertices on its top and bottom edges, numbered $1,\dots, r+s$ in order from left to right. Each vertex is connected by a strand to exactly one other vertex. In addition, there is a vertical wall separating the left $r$ vertices from the right $s$ vertices, such that the following conditions hold: 
	\begin{enumerate} 
		\item A \emph{propagating line} connects a vertex on the top row with one on the bottom row, and it cannot cross the wall.
		\item A \emph{northern arc} (respectively, \emph{southern arc}) connects vertices on the top row (respectively, on the bottom row), and it must cross the wall.
		\end{enumerate} 
	For example, the following graphs are $(4,2)$-Brauer diagrams:
		\[	x:=\vcenter
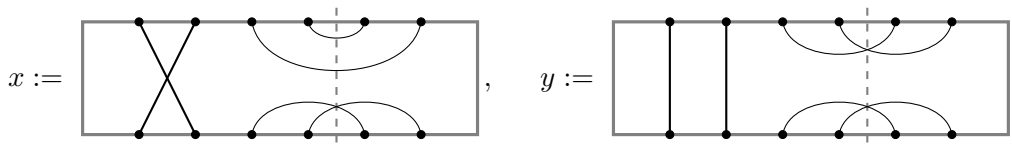
{\hbox{
			\begin{tikzpicture}[scale=0.75]
				\draw[gray,very thick] (0,0) rectangle (7,2);
				\foreach \x in {1,2,3,4,5,6} \filldraw (\x,2) circle (2pt);
				\foreach \x in {1,2,3,4,5,6} \filldraw (\x,0) circle (2pt);
				\draw[gray, thick, dashed] (4.5,2.25) to (4.5, -0.25);
				\draw[thick] (1,2) to (2,0);
				\draw[thick] (2,2) to (1,0);
				\draw (3,2) to [bend right=80] (6,2);
					\draw (4,2) to [bend right=80] (5,2);
				\draw (3,0) to [bend left=80] (5,0);
					\draw (4,0) to [bend left=80] (6,0);
	\end{tikzpicture}}}\,,
	\hspace{1.5em}
	y:=\vcenter{\hbox{
			\begin{tikzpicture}[scale=0.75]
				\draw[gray,very thick] (0,0) rectangle (7,2);
				\foreach \x in {1,2,3,4,5,6} \filldraw (\x,2) circle (2pt);
				\foreach \x in {1,2,3,4,5,6} \filldraw (\x,0) circle (2pt);
				\draw[gray, thick, dashed] (4.5,2.25) to (4.5, -0.25);
				\draw[thick] (1,2) to (1,0);
				\draw[thick] (2,2) to (2,0);
				\draw (3,2) to [bend right=80] (5,2);
				\draw (4,2) to [bend right=80] (6,2);
				\draw (3,0) to [bend left=80] (5,0);
				\draw (4,0) to [bend left=80] (6,0);
	\end{tikzpicture}}}\]

	Let $\delta$ be a complex number. The \emph{walled Brauer algebra} $B_{r,s}(\delta)$ is the $\mathbb{C}$-linear span of the $(r,s)$-walled Brauer diagrams with the multiplication
	defined as follows: The product of two $(r,s)$-walled Brauer diagrams $d_1$ and $d_2$ is determined by putting $d_1$ above $d_2$ and identifying the bottom vertices of $d_1$ with the top vertices of $d_2$. Let $n$ be the number of
	 loops in the middle row so obtained.  The product $d_1d_2$ is given by $\delta^n$ times the resulting diagram with loops omitted. 
	 
	 For example, for the diagrams $x$ and $y$ above we have:
	 \[	xy=\delta^2 \vcenter
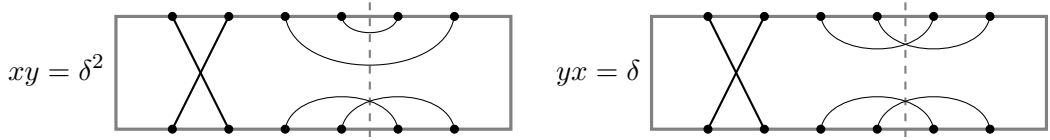
{\hbox{
 \begin{tikzpicture}[scale=0.75]
 \draw[gray,very thick] (0,0) rectangle (7,2);
 \foreach \x in {1,2,3,4,5,6} \filldraw (\x,2) circle (2pt);
 \foreach \x in {1,2,3,4,5,6} \filldraw (\x,0) circle (2pt);
 \draw[gray, thick, dashed] (4.5,2.25) to (4.5, -0.25);
 \draw[thick] (1,2) to (2,0);
 \draw[thick] (2,2) to (1,0);
 \draw (3,2) to [bend right=80] (6,2);
 \draw (4,2) to [bend right=80] (5,2);
 \draw (3,0) to [bend left=80] (5,0);
 \draw (4,0) to [bend left=80] (6,0);
\end{tikzpicture}}}
	 \hspace{1.5em}
	 yx=\delta\vcenter{\hbox{
	 		\begin{tikzpicture}[scale=0.75]
	 			\draw[gray,very thick] (0,0) rectangle (7,2);
	 			\foreach \x in {1,2,3,4,5,6} \filldraw (\x,2) circle (2pt);
	 			\foreach \x in {1,2,3,4,5,6} \filldraw (\x,0) circle (2pt);
	 			\draw[gray, thick, dashed] (4.5,2.25) to (4.5, -0.25);
	 			\draw[thick] (1,2) to (2,0);
	 			\draw[thick] (2,2) to (1,0);
	 			\draw (3,2) to [bend right=80] (5,2);
	 			\draw (4,2) to [bend right=80] (6,2);
	 			\draw (3,0) to [bend left=80] (5,0);
	 			\draw (4,0) to [bend left=80] (6,0);
	 \end{tikzpicture}}}\]
	 The dimension of  $B_{r,s}(\delta)$  equals 
	 $(r+s)!$ (see, for example,  \cite[2.2]{BS}).
	 
	 We denote by $s_i\, (1\leq i \leq r-1 \text{ or } r+1\leq i\leq r+s-1)$ and $e$ the following  $(r,s)$-walled Brauer diagrams: 
	$$s_i:=\vcenter
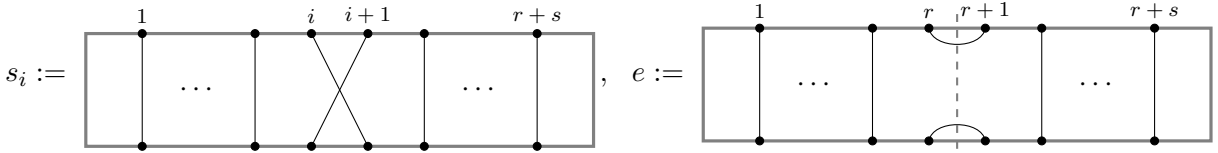
{\hbox{
			\begin{tikzpicture}[scale=0.75]
	 	\draw[gray, very thick] (0,0) rectangle (9,2);
	 	\foreach \x in {1,3,4,5,6,8} \filldraw (\x,0) circle (2pt);
	 	\foreach \x in {1,3,4,5,6,8} \filldraw (\x,2) circle (2pt);
	 	\draw (1,2) node[above]{\scriptsize $1$};
	 	\draw (8,2) node[above]{\scriptsize $r+s$};
	 	\draw (4,0) to (5,2);
	 	\draw (4,2) to (5,0);
	 	\draw (2,1) node{$\cdots$};
	 	\draw (7,1) node{$\cdots$};
	 	\foreach \x in {1,3,6,8} \draw (\x,0) to (\x,2);
	 	\draw (4,2) node[above]{\scriptsize $i$};
	 	\draw (5,2) node[above]{\scriptsize $i+1$};
	 \end{tikzpicture}}}\,, \,\,\,\,
	 e:=\vcenter{\hbox{
	 		\begin{tikzpicture}[scale=0.75]
	 	\draw[gray, very thick] (0,0) rectangle (9,2);
	 	\foreach \x in {1,3,4,5,6,8} \filldraw (\x,0) circle (2pt);
	 	\foreach \x in {1,3,4,5,6,8} \filldraw (\x,2) circle (2pt);
	 	\draw (1,2) node[above]{\scriptsize $1$};
	 	\draw (8,2) node[above]{\scriptsize $r+s$};
	 	\draw (4,0) to [bend left=80] (5,0);
	 	\draw (4,2) to [bend right=80] (5,2);
	 	\draw (2,1) node{$\cdots$};
	 	\draw (7,1) node{$\cdots$};
	 	\foreach \x in {1,3,6,8} \draw (\x,0) to (\x,2);
	 	\draw (4,2) node[above]{\scriptsize $r$};
	 	\draw (5,2) node[above]{\scriptsize $r+1$};
	 	\draw[gray, thick, dashed] (4.5,2.25) to (4.5, -0.25);
	 \end{tikzpicture}}}$$
 Note that $B_{0,n}(\delta)\simeq B_{n,0}(\delta)\simeq \mathbb{C}S_n$, the group algebra of the symmetric group $S_n$ on $n$ letters. 
 It's easy to check that the algebra $B_{r,s}(\delta)$ is generated by the elements $s_i\,(1\leq i \leq r-1 \text{ or } r+1\leq j\leq r+s-1)$ and $e$. 

For any $(r,s)$-walled Brauer diagram $d$, we denote by $d^*$ the \emph{flip diagram} of $d$,  obtained by horizontally flipping $d$. This define a $\mathbb{C}$-linear anti-automorphism $*: B_{r,s}(\delta)\rightarrow B_{r,s}(\delta)$.

The following theorem is Theorem 6.3 in \cite{CDDM} and it provides a criterion of the semisimplicity of the algebra $B_{r,s}(\delta)$.

\begin{thm}\label{ss} The walled Brauer algebra $B_{r,s}(\delta)$ is semisimple if and only if one of the following holds:
	\begin{enumerate}
		\item $r=0$ or $s=0$,
		\item $\delta \not \in \mathbb{Z}$, 
		\item $|\delta|> r+s-2$,
		\item $\delta=0$, and $(r,s)\in \{(1,2),\, (1,3),\,(2,1),\,(3,1)\}.$
		\end{enumerate} 
\end{thm}
Therefore, for a fixed pair $(r,s)$ the algebra $B_{r,s}(\delta)$ is semisimple except for finitely many values $\delta \in \mathbb{C}$.

\subsection{Jucys-Murphy elements}\label{jm}
There are different definitions for Jucys-Murphy elements for the walled Brauer algebra $B_{r,s}(\delta)$. In this paper, we use Definition 2.1 of \cite{JK}.

Consider the transposition $(a,b)$ given by the diagram
$$\begin{array}{lcl}\vcenter
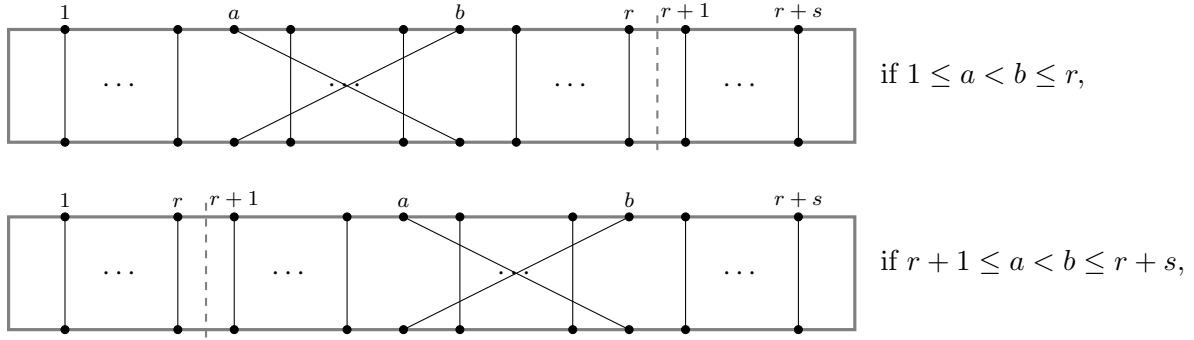
{\hbox{
		\begin{tikzpicture}[scale=0.75]
			\draw[gray, very thick] (0,0) rectangle (15,2);
			\foreach \x in {1,3,4,5,7,8,9,11,12,14} \filldraw (\x,0) circle (2pt);
			\foreach \x in {1,3,4,5,7,8,9,11,12,14} \filldraw (\x,2) circle (2pt);
			\draw (2,1) node{$\cdots$};
			\draw (6,1) node{$\cdots$};
				\draw (10,1) node{$\cdots$};
					\draw (13,1) node{$\cdots$};
					\foreach \x in {1,3,5,7,9,11,12,14} \draw (\x,0) to (\x,2);
					\draw (4,0) to (8,2);
					\draw (4,2) to (8,0);
					\draw (1,2) node[above]{\scriptsize $1$};
			\draw (11,2) node[above]{\scriptsize $r$};
			\draw (12,2) node[above]{\scriptsize $r+1$};
			\draw (14,2) node[above]{\scriptsize $r+s$};
			\draw (4,2) node[above]{\scriptsize $a$};
			\draw (8,2) node[above]{\scriptsize $b$};
				\draw[gray, thick, dashed] (11.5,2.25) to (11.5, -0.25);
\end{tikzpicture}}}\,\,\,\text{ if } 1\leq a<b\leq r, &&\smallbreak\smallbreak\smallbreak\smallbreak\\
\,\,\vcenter{\hbox{\begin{tikzpicture}[scale=0.75]
			\draw[gray, very thick] (0,0) rectangle (15,2);
			\foreach \x in {1,3,4,6,7,8,10,11,12,14} \filldraw (\x,0) circle (2pt);
			\foreach \x in {1,3,4,6,7,8,10,11,12,14} \filldraw (\x,2) circle (2pt);
			\draw (2,1) node{$\cdots$};
			\draw (5,1) node{$\cdots$};
			\draw (9,1) node{$\cdots$};
			\draw (13,1) node{$\cdots$};
			\foreach \x in {1,3,4,6,8,10,12,14} \draw (\x,0) to (\x,2);
			\draw (7,0) to (11,2);
			\draw (7,2) to (11,0);
			\draw (1,2) node[above]{\scriptsize $1$};
			\draw (3,2) node[above]{\scriptsize $r$};
			\draw (4,2) node[above]{\scriptsize $r+1$};
			\draw (14,2) node[above]{\scriptsize $r+s$};
			\draw (7,2) node[above]{\scriptsize $a$};
			\draw (11,2) node[above]{\scriptsize $b$};
			\draw[gray, thick, dashed] (3.5,2.25) to (3.5, -0.25);
\end{tikzpicture}}}\,\,\, \text{ if } r+1\leq a<b\leq r+s, &&
\end{array}$$
Define also $e_{j,k},\, (1\leq j \leq r,\, r+1\leq k\leq r+s)$ to be the diagram 

$$\vcenter
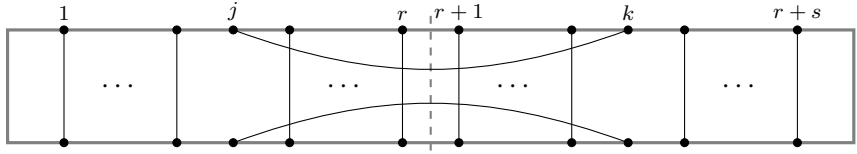
{\hbox{
		\begin{tikzpicture}[scale=0.75]
			\draw[gray, very thick] (0,0) rectangle (15,2);
			\foreach \x in {1,3,4,5,7,8,10,11,12,14} \filldraw (\x,0) circle (2pt);
			\foreach \x in {1,3,4,5,7,8,10,11,12,14} \filldraw (\x,2) circle (2pt);
			\draw (2,1) node{$\cdots$};
			\draw (6,1) node{$\cdots$};
			\draw (9,1) node{$\cdots$};
			\draw (13,1) node{$\cdots$};
			\foreach \x in {1,3,5,7,8,10,12,14} \draw (\x,0) to (\x,2);
			\draw (1,2) node[above]{\scriptsize $1$};
			\draw (7,2) node[above]{\scriptsize $r$};
			\draw (8,2) node[above]{\scriptsize $r+1$};
			\draw (14,2) node[above]{\scriptsize $r+s$};
			\draw (4,2) node[above]{\scriptsize $j$};
			\draw (11,2) node[above]{\scriptsize $k$};
			\draw[gray, thick, dashed] (7.5,2.25) to (7.5, -0.25);
			\draw (4,0) to [bend left=20] (11,0);
			\draw (4,2) to [bend right=20] (11,2);
\end{tikzpicture}}}$$

For each $1\leq k\leq r+s$ we define the \emph{Jucys-Murphy elements}  $L_k$ of $B_{r,s}(\delta)$ as follows:
$$L_k:=\begin{cases}
	0 &\text { if } k=1,\\
	\sum\limits_{j=1}^{k-1}(j,k) & \text{ if } 1<k\leq r,\\
		-\sum\limits_{j=1}^{r}e_{j,k}+\sum\limits_{j=r+1}^{k-1}(j,k)+\delta & \text{ if } r+1\leq k\leq r+s.
\end{cases}$$

\subsection{Supersymmetric polynomials and central elements} 

Let $m,n$ be nonnegative integers. An element $p$ in the polynomial ring $\mathbb{C}[x_1,\dots, x_m, y_1,\dots,y_n]$ is supersymmetric if
\begin{enumerate}
	\item $p$ is symmetric in $x_1,\dots, x_m$ and in $y_1,\dots, y_n$ separately.
	\item The substitution $x_m=t$, $y_1=-t$ yields a polynomial in $x_1,\dots, x_{m-1}, y_2,\dots, y_m$ which is independent of $t$.
\end{enumerate}
		We denote by $\mathcal{S}_{m,n}[x;y]$ the set of supersymmetric polynomials in $x_1,\dots, x_m, y_1,\dots,y_n$.
		
		The following result is Corollary 2.9 in \cite{JK} and it provides us with some central elements of $B_{r,s}(\delta)$.
		\begin{prop}\label{pjm}For every supersymmetric polynomial $p$ in $\mathcal{S}_{r,s}[x;y]$, the element $$p(L_1,\dots, L_{r}, L_{r+1},\dots, L_{r+s})$$ belongs to the centre of $B_{r,s}(\delta)$.
		\end{prop}

\subsection{The semisimple case}
		We know that when an algebra is semisimple, the dimension of its centre is equal to the number of isomorphic classes of simple modules. For the walled Brauer algebra $B_{r,s}(\delta)$, these are indexed by the set $\Lambda_{r,s}$ of bipartitions defined as follows (see for example \cite[Theorem 2.7]{CDDM}).

		A partition is a decreasing sequence of non-negative integers $\lambda=(\lambda_1, \lambda_2,\dots)$. We write $\Lambda$ for the set of all partitions. For $\lambda \in \Lambda$ we set $|\lambda| := \sum\limits_{i\geq 1}\lambda_i$.   Now define
		$$\Lambda_{r,s}:=\coprod\limits_{k=0}^{\min(r,s)}\{(\lambda, \mu)\in \Lambda \times \Lambda \,|\, |\lambda|=r-k, \,|\mu|=s-k\}$$

		Therefore, when $B_{r,s}(\delta)$ is semisimple, the dimension of its centre is the cardinality of $\Lambda_{r,s}$. Jung and Kim \cite[Theorem 3.5]{JK} gave a basis of the centre of $B_{r,s}(\delta)$ in this case. More precisely, they proved the following theorem.
		\begin{thm}
			If the walled Brauer algebra $B_{r,s}(\delta)$ is semisimple, the supersymmetric polynomials in $L_1,\dots, L_{r+s}$ generate the centre of $B_{r,s}(\delta)$. Moreover, there is a set of supersymmetric polynomials $\{p_{(\lambda, \mu)}\in S_{r,s}[x;y]\,|\, (\lambda,\mu) \in \Lambda_{r,s}\}$, such that  the set $\{p_{(\lambda, \mu)}(L_1,\dots, L_{r+s})\,|\, (\lambda, \mu) \in \Lambda_{r,s}\}$ is a basis of the centre of $B_{r,s}(\delta)$.
		\end{thm}
	One can find a precise description of these supersymmmetric polynomials $p_{\lambda}$ in \cite[Section 3]{JK}.
	
	The goal of this paper to extend this result to the non-semisimple case. More precisely, we want to prove the following conjecture \cite[Conjecture 5.4]{JK}:
	
	\begin{conj}\label{mainconjecture}
		For every $\delta \in \mathbb{C}$, the center of the walled Brauer algebra $B_{r,s}(\delta)$ is generated by the supersymmetric polynomials in the Jucys-Murphy elements $L_1, \dots, L_{r+s}$.
	\end{conj}

\section{A lower bound on the dimension of the centre of $B_{r,s}(\delta)$}

Let $z$ be a formal parameter. We consider the walled Brauer algebra $B_{r,s}^{\mathbb{C}[z]}(z)$ over the polynomial ring $\mathbb{C}[z]$.
We define the Jucys-Murphy elements $\mathcal{L}_i$, $1\leq i\leq r+s$ as in Section \ref{jm}, replacing $\delta$ by $z$. The proof of Proposition \ref{pjm} goes through unchanged when working in this setting and we get the following proposition.

\begin{prop}\label{pjm2}
	All supersymmetric polynomials (with coefficients in $\mathbb{C}[z]$) in $\mathcal{L}_1,\dots, \mathcal{L}_{r+s}$ are central in $B_{r,s}^{\mathbb{C}[z]}(z)$.
\end{prop}

We now take a $\delta_0\in \mathbb{C}$ ($\delta_0\not=0$), such that $B_{r,s}(\delta_0)$ is semisimple (see Theorem \ref{ss}). Set $m=|\Lambda_{r,s}|$ and choose $p_i\in S_{r,s}[x;y]$ for $1\leq i\leq m$ to be supersymmetric polynomials such that 
$$\{p_i(L_1,\dots, L_{r+s})\,|\, 1\leq i\leq m\}$$
form a $\mathbb{C}$-basis for the centre $Z(B_{r,s}(\delta_0))$ of $B_{r,s}(\delta_0)$.

Let $\mathbb{C}(z)$ be the field of fractions of $\mathbb{C}[z]$ and we consider the algebra
$$B_{r,s}^{\mathbb{C}(z)}(z):=B_{r,s}^{\mathbb{C}[z]}(z)\otimes_{\mathbb{C}[z]}\mathbb{C}(z).$$

Using the same arguments of Theorem 6.3 in \cite{CDDM}, we have that $B_{r,s}^{\mathbb{C}(z)}(z)$ is a semisimple $\mathbb{C}(z)$-algebra with $m$ simple modules, hence
$$\dim_{\mathbb{C}(z)}Z(B_{r,s}^{\mathbb{C}(z)}(z))=m.$$
		Moreover, 
		$$Z(B_{r,s}^{\mathbb{C}(z)}(z))=Z(B_{r,s}^{\mathbb{C}[z]}(z))\otimes_{\mathbb{C}[z]}\mathbb{C}(z).$$
		Therefore, 
		$$\rank_{\mathbb{C}[z]}Z(B_{r,s}^{\mathbb{C}[z]}(z))=\dim_{\mathbb{C}(z)}Z(B_{r,s}^{\mathbb{C}(z)}(z))=m.$$
		
		\begin{prop}
			The set $\{p_i(\mathcal{L}_1,\dots, \mathcal{L}_{r+s})\,|\, 1\leq i\leq m\}$ forms a $\mathbb{C}[z]$-basis for $Z(B_{r,s}^{\mathbb{C}[z]}(z))$.
		\end{prop}
	\begin{proof}
		From Proposition \ref{pjm2} we have that $p_i(\mathcal{L}_1,\dots, \mathcal{L}_{r+s})\in
		Z(B_{r,s}^{\mathbb{C}[z]}(z))$, for every $i=1,\dots,m$. Moreover,  $\rank_{\mathbb{C}[z]}Z(B_{r,s}^{\mathbb{C}[z]}(z))=m$. Therefore, it is enough to show that this set is linearly independent.
		
		Assume $\sum\limits_{i=1}^m a_i(z)p_i(\mathcal{L}_1,\dots, \mathcal{L}_{r+s})=0$, for some $a_i(z)\in \mathbb{C}[z]$, not all 0. 
			Dividing by the highest power of $(z-\delta_0)$ which divides all the $a_i(z)$'s, we can assume that there exists some $i_0$ with $a_{i_0}(\delta_0)\not=0$. 
		
		Now, specializing to $z=\delta_0$, and noting that under this specialization the elements $\mathcal{L}_i$'s become the $L_i$'s, we get
		$$\sum\limits_{i=1}^m a_i(\delta_0)p_i(L_1,\dots,L_{r+s})=0.$$
		Now, as $\{p_i(L_1,\dots, L_{r+s})\,|\, 1\leq i\leq m\}$ is a linearly independent set, we must have $a_i(\delta_0)=0,\, \forall i=1,\dots,m$, which contradicts the fact that $a_{i_0}(\delta_0)\not=0$. 
	\end{proof}
		
		We now take an arbitrary $\delta\in \mathbb{C}$ and we consider
		$$B^{\mathbb{C}}_{r,s}(\delta)=B_{r,s}^{\mathbb{C}[z]}(z)\otimes_{\mathbb{C}[z]}\mathbb{C}[z]/\langle  z-\delta\rangle.$$
		
		For any $Q\in B_{r,s}^{\mathbb{C}[z]}(z)$, we write 
		$$\bar{Q}:=Q\otimes_{\mathbb{C}[z]}\mathbb{C}[z]/\langle  z-\delta\rangle\in B^{\mathbb{C}}_{r,s}(\delta).$$
		
		To simplify notation we write
		$$\mathcal{P}_i:=p_i(\mathcal{L}_1,\dots, \mathcal{L}_{r+s})\in
		 B_{r,s}^{\mathbb{C}[z]}(z),\,\,i=1,\dots, m.$$
		 
		 \begin{prop}\label{gc} 
		 	For any choice of $\delta \in \mathbb{C}$, the set $\{\bar{\mathcal{P}_i}\,|\, 1\leq i\leq m\}$ is a linearly independent set in $Z(B_{r,s}^{\mathbb{C}}(z))$. In particular, $\dim_{\mathbb{C}}Z(B_{r,s}^{\mathbb{C}}(z))\geq |\Lambda_{r,s}|$.
		 \end{prop}
	 
	 \begin{proof}
	 	First note that as $\mathcal{P}_i$ are central elements in $B_{r,s}^{\mathbb{C}[z]}(z)$ we have that  $\bar{\mathcal{P}_i}$ are central in $B_{r,s}^{\mathbb{C}}(z)$.
	 	Now let $$\sum\limits_{i=1}^m a_i\bar{\mathcal{P}}_i=0$$
	 	for some $a_i\in \mathbb{C}$. Then, we have
	 	$$\sum\limits_{i=1}^m a_i\bar{\mathcal{P}_i}=\sum\limits_{i=1}^m \overline{a_i\mathcal{P}_i}=0$$ and so 	
	 	\begin{equation}\label{eqq}
	 	\sum\limits_{i=1}^m a_i\mathcal{P}_i=(z-\delta)R
	 	\end{equation} 
 	for some $R\in Z(B_{r,s}^{\mathbb{C}[z]}(z))$. Thus, $R$ can be written as $R=\sum\limits_{i=1}^m b_i(z)\mathcal{P}_i$, for some $b_i(z)\in \mathbb{C}[z]$. Using now Equation \eqref{eqq} we get 
 	$\sum\limits_{i=1}^m a_i\mathcal{P}_i=\sum\limits_{i=1}^m(z-\delta)b_i(z)\mathcal{P}_i$. As $a_i\in \mathbb{C}$, $b_i(z)\in \mathbb{C}[z]$ and $\{\mathcal{P}_i\,|\, 1\leq i\leq m\}$ are linearly independent, we must have $a_i=(z-\delta)b_i(z)=0$, $\forall i=1,\dots, m$.
	 \end{proof}

\section{Shalile's cycle type and general strategy}

\subsection{Walled generalized cycle types} The following definition is Definition 7.3 in \cite{Sha}. For a diagram $d\in B_{r,s}(\delta)$  we define the  \emph{walled generalized cycle type} $c(d)$ of $d$ to be a set of words (called \emph{parts}) in the alphabet $L, R, N$ and $S$, obtained as follows: We first connect each vertex in
the top row of $d$ with the vertex in the bottom row below it. The parts of $c(d)$ correspond to the connected components of this new graph, as follows:  We take a connected component of the new graph, we pick a vertex of it and we follow the path, until all the edges of the connected component have been read off once.  Following the path, we record in order with the letters $L, R, N$ and $S$ the types of the edges of the  diagram $d$ which are traversed. More precisely, we record:
\begin{itemize}
	\item $N$ if the type of the edge is a northern arc.
	\item $S$ if the type of the edge is a southern arc.
	\item $L$ if the type of the edge is a propagating line to the left of the wall.
		\item $R$ if the type of the edge is a propagating line to the right of the wall.
\end{itemize}

For example, for the diagrams $x$ and $y$ we saw in Section \ref{def} we have 
$c(x)=\{LL, NSNS\}$ and $c(y)=\{L,L, NS,NS\}$.

Two parts are \emph{equivalent} if one is obtained from the other by repeated cyclic permutation and/or reverse reading. Two walled generalized cycle types are \emph{equal} if their parts are equivalent. For example, we also have $c(y)=\{L,L,SN, SN\}$ and $c(x)=\{LL, SNSN\}$.

\begin{rem}\label{ct} From the definition of (equal) walled generalized cycle types we notice the following: 
	\begin{enumerate}
\item	Each part of $c(d)$ is of one of the following forms:
	\begin{enumerate}
		\item $L^a$, $a\geq 1$.
		\item $R^b$, $b\geq 1$.
		\item $NR^{b_1}SL^{a_1}NR^{b_2}SL^{a_2}\dots$, \, for some $a_i, b_i\geq 0$.
			\end{enumerate}
		\item The number of $N$'s in $c(d)$ equals the number of $S$'s in $c(d)$. We denote this number by $t$. Then, we also have that the number of $L$'s in $c(d)$ equals $r-t$, while the  number of $R$'s in $c(d)$ equals $s-t$.
		\item In general, $c(d)\not=c(d^*)$. For example, we consider the following diagram $d$ of the walled Brauer algebra $B_{3,3}(\delta)$:
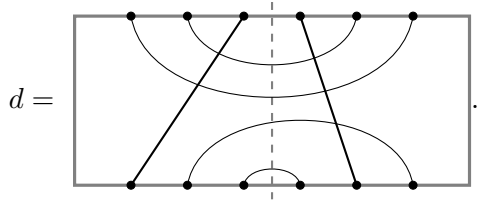
			$$d=\vcenter{\hbox{
				\begin{tikzpicture}[scale=0.75]
				\draw[gray,very thick] (0,0) rectangle (7,3);
				\foreach \x in {1,2,3,4,5,6} \filldraw (\x,3) circle (2pt);
				\foreach \x in {1,2,3,4,5,6} \filldraw (\x,0) circle (2pt);
				\draw[gray, thick, dashed] (3.5,3.25) to (3.5, -0.25);
				\draw (3,0) to [bend left=80] (4,0);
				\draw (1,3) to [bend right=80] (6,3);
				\draw (2,0) to [bend left=80] (6,0);
				\draw (2,3) to [bend right=80] (5,3);
				\draw[thick] (3,3) to (1,0);
				\draw[thick] (4,3) to (5,0);
			\end{tikzpicture}}}.$$ We have $c(d)=\{NSNRSL\}$. On the other hand, 
		$$d^*=\vcenter
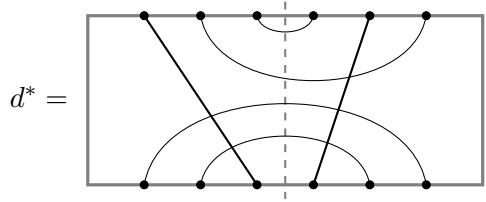
{\hbox{
				\begin{tikzpicture}[scale=0.75]
				\draw[gray,very thick] (0,0) rectangle (7,3);
				\foreach \x in {1,2,3,4,5,6} \filldraw (\x,3) circle (2pt);
				\foreach \x in {1,2,3,4,5,6} \filldraw (\x,0) circle (2pt);
				\draw[gray, thick, dashed] (3.5,3.25) to (3.5, -0.25);
				\draw (3,3) to [bend right=80] (4,3);
				\draw (1,0) to [bend left=80] (6,0);
				\draw (2,3) to [bend right=80] (6,3);
				\draw (2,0) to [bend left=80] (5,0);
				\draw[thick] (1,3) to (3,0);
				\draw[thick] (5,3) to (4,0);
			(	
			\end{tikzpicture}}}$$
		We have $c(d^*)=\{NSLNRS\}$. Note that  $c(d)\not=c(d^*)$. A description of how to obtain $c(d^*)$ from $c(d)$ in general is given in \cite[Lemma 2.14]{Sha}.
	\end{enumerate} 
\end{rem}
The following theorem is Theorem 7.7 in \cite{Sha}.
\begin{thm}\label{conj}
Two diagrams of $B_{r,s}(\delta)$ are $S_{r}\times S_{s}$-conjugate if and only if they have equal walled generalized cycle types.
\end{thm}

\subsection{A general strategy}
Let $\mathcal{B}$ be the diagram basis of $B_{r,s}(\delta)$. We denote by $C_{r,s}$ the set of all generalised walled cycle types for $B_{r,s}(\delta)$.
For each walled generalized cycle type $\mu\in C_{r,s}$ we set $\mathcal{B}_{\mu}:=\{d\in \mathcal{B}\,|\, c(d)=\mu\}$. 
 It follows from Theorem \ref{conj} that the centralizer of $Z_{B_{r,s}(\delta)}(\mathbb{C}(\mathfrak{S}_r\times \mathfrak{S}_s))$ has a basis consisting of all elements of the form   $\sum\limits_{d\in \mathcal{B}_{\mu}}d$ where $\mu$ runs over all walled generalised cycle types. 
 Recall from Section \ref{def} that the walled Brauer algebra $B_{r,s}(\delta)$ is generated by
 the generators of  $S_r\times S_s$ and by the generator $e$.
  A central element of $B_{r,s}(\delta)$ is therefore an element of $Z_{B_{r,s}(\delta)}(\mathbb{C}(S_r\times S_s))$ which commutes with the generator $e$. Let us consider an arbitrary element of $Z_{B_{r,s}(\delta)}(\mathbb{C}(S_r\times S_s))$, which is of the form $\sum\limits_{\mu}a_{\mu} \sum\limits_{d\in \mathcal{B}_{\mu}}d$, for some $a_{\mu}\in \mathbb{C}$. In order for this element to be in the centre of $B_{r,s}(\delta)$, it needs to satisfy
	\begin{equation} \label{eq}
		\sum\limits_{\mu}a_{\mu} \sum\limits_{d\in \mathcal{B}_{\mu}}(de-ed)=\sum\limits_{x\in \mathcal{B}}b_xx = 0.
	\end{equation} 	
	So we need to consider the system 
	\begin{equation}\label{sys}
	\{b_x = 0 \, | \, x\in \mathcal{B}\}
	\end{equation}
	as a system of equations in the variables $a_\mu$,  $\mu\in C_{r,s}$.  Showing that the rank of this system is $|C_{r,s}|-|\Lambda_{r,s,}|$ would prove Conjecture \ref{mainconjecture}. 
In fact, using Proposition \ref{gc}, it is enough to show that the rank is at least $|C_{r,s}|-|\Lambda_{r,s}|$. Thus our strategy is to find $|C_{r,s}|-|\Lambda_{r,s}|$ linearly independent equations among $\{b_x = 0 \, | \, x\in \mathcal{B}\}$.  To identify which equations we should be considering, let us start by making the following observation.

		\begin{defn} Let $\mu\in C_{r,s}$. We call a part of $\mu$ of type $(a)$, $(b)$, or $(c)$ of the form $NS$ in Remark \ref{ct}(1) a \emph{trivial part}.
			\end{defn}
			
				\begin{prop}\label{bij} 
			There is a bijection between the set $\Lambda_{r,s}$ and the set of diagrams in $B_{r,s}(\delta)$ having only trivial parts.
		\end{prop}
		\begin{proof} Let $k\in \mathbb{N}_0$ and let $\lambda=(\lambda_1, \lambda_2, \dots)$,  $\mu=(\mu_1,\mu_2,\dots)$ partitions of $r-k$ and $s-k$, respectively. The following map is a bijection:
			$$(\lambda, \mu)\mapsto \left\{d\in B_{r,s}(\delta)\,\,|\,\, c(d)=\{L^{\lambda_1},\,L^{\lambda_2},\dots, R^{\mu_1},\,R^{\mu_2},\dots, \underbrace{NS,\, NS,\dots, NS}_{k \text { terms}}\}\right\}.$$
			\end{proof}

Thus in order to prove Conjecture \ref{mainconjecture}, it would be enough to pick one diagram $x\in \mathcal{B}_\mu$ for each cycle type $\mu\in C_{r,s}$ containing at least one non-trivial part and show that the set of equations $b_x=0$ for all such $x$'s are linearly independent. We will do this in the case $s=1$ in the next section.

\section{The center of $B_{r,1}(\delta)$}
Let $\mathcal{B}$ be the diagram basis of $B_{r,1}(\delta)$, consisting of $(r+1)!$ diagrams. We recall that for each walled generalized cycle type $\mu$ we set $\mathcal{B}_{\mu}=\{d\in \mathcal{B}\,|\, c(d)=\mu\}$.

The following lemma is easy to see but note that  it is true only for the case of $B_{r,1}(\delta)$ (see Remark \ref{ct} (3)).
\begin{lem}\label{flip}
	The flip map $*$ gives a bijection $\mathcal{B}_{\mu}\rightarrow \mathcal{B}_{\mu}$, $d\mapsto d^*$ for each $\mu\in C_{r,s}$.
\end{lem}
We now consider equation \eqref{eq} and system \eqref{sys}. Our first attempt to simplify this system is to find diagrams $x\in B_{r,1}(\delta)$ for which we have $b_x=0$.
The first result in this direction is the following proposition:

\begin{prop}\label{pro}
For any $x\in \mathcal{B}$ we have	$b_{x^*}=-b_x$. In particular, if $x=x^*$ then $b_x=0$.
\end{prop}
\begin{proof}
	We apply the anti-automorphism $*$ to both sides of \eqref{eq} and we have:
	$$\begin{array}{lcl}
			\sum\limits_{\mu}a_{\mu} \sum\limits_{d\in C_{\mu}}(de-ed)^*&=&\sum\limits_{x\in \mathcal{B}}b_xx^*\,\, \Longrightarrow \smallbreak\smallbreak\\
			\sum\limits_{\mu}a_{\mu} \sum\limits_{d\in C_{\mu}}(e^*d^*-d^*e^*)&=&\sum\limits_{x\in \mathcal{B}}b_xx^*\,\,\stackrel{e^*=e}{\Longrightarrow} \smallbreak\smallbreak\\
			\sum\limits_{\mu}a_{\mu} \sum\limits_{d\in C_{\mu}}(ed^*-d^*e)&=&\sum\limits_{x\in \mathcal{B}}b_xx^*\,\,\stackrel{\ref{flip}}{\Longrightarrow} \smallbreak\smallbreak\\
				\sum\limits_{\mu}a_{\mu} \sum\limits_{d\in C_{\mu}}(ed-de)&=&\sum\limits_{x\in \mathcal{B}}b_xx^*\,\,\Longrightarrow \smallbreak\smallbreak\\
		-\sum\limits_{\mu}a_{\mu} \sum\limits_{d\in C_{\mu}}(de-ed)&=&	\sum\limits_{x\in \mathcal{B}}b_xx^*\,\,\Longrightarrow \smallbreak\smallbreak\\
			-\sum\limits_{x\in \mathcal{B}}b_xx&=&\sum\limits_{x\in \mathcal{B}}b_xx^*\,\,\Longrightarrow \smallbreak\smallbreak\\
				-\sum\limits_{x\in \mathcal{B}}b_xx&=&\sum\limits_{x\in \mathcal{B}}b_{x^*}x.
	\end{array}$$
\end{proof}
The next proposition is another case of a basis diagram $x$, such that $b_x=0$.
\begin{prop}\label{pr1}
	Let $x=\vcenter{\hbox{
			\begin{tikzpicture}[scale=0.75]
				\draw[gray,very thick] (0,0) rectangle (5,3);
					\draw[gray,very thick] (0.3,0.3) rectangle (2.6,2.6)node[pos=.5] {$\mathcal{Q}$};
				
				\foreach \x in {3,4} \filldraw (\x,3) circle (2pt);
				\foreach \x in {3,4} \filldraw (\x,0) circle (2pt);
				\draw[gray, thick, dashed] (3.5,3.25) to (3.5, -0.25);
				\draw (3,3) to [bend right=80] (4,3);
				\draw (3,0) to [bend left=80] (4,0);

	\end{tikzpicture}}}\in B_{r,1}(\delta)$ for some $\mathcal{Q}\in S_{r-1}$ then $b_x=0$.
\end{prop}
\begin{proof}
	Let $d\in \mathcal{B}$ such that $de=\delta^mx$, $m\in\{0,1\}$. We have the following cases:
	\begin{enumerate}
		\item $d=\mathcal{Q}\in S_{r-1}\subseteq B_{r,1}(\delta)$, with $\mathcal{Q}$ as defined in the diagram of $x$. We have $de=x$.
		\item $d=x$. We have $de=\delta x$. 
		\item $d=x(i,\,r)$, where $(i,\,r)$ corresponds to the permutation diagram swapping $i$ and $r$. We have $de=x$. 
		\end{enumerate}
	
	We notice that for every diagram $d$ as described above, there is a diagram $d'\in \mathcal{B}$ such that $ed'=\delta^mx$, $m\in\{0,1\}$. More precisely, for the cases (1) and (2), the diagram $d'=d$. For case (3), we have $d'=(i,\, r)x$. We notice that the diagrams $d'$ we have here are all the diagrams $d'\in \mathcal{B}$ such that $de=\delta^mx$, $m\in\{0,1\}$. In order to prove that $b_x=0$ it is enough to prove that in each case, $d$ and $d'$ have the same walled generalized cycle type. Then the terms in the sum in equation \eqref{eq} cancel out and, hence, $b_x=0$. For cases (1) and (2) this is obvious. It remains to prove that the diagrams $x(i,\,r)$ and $(i,\,r)x$ have the same walled generalized cycle type.
	
	There are two types of parts in the walled generalized cycle type $c(x)$ of the diagram $x$:  The part $p_1=NS$ and the parts $p_j$, $j=2,3,\dots,k$, which are of 
	one of the forms: $L$, $LL$, $LLL$, $\dots$. Let $p_{j_0}$ be the part, which belongs to the connected component with vertex $i$ on top.
	
	In the diagram of $d$ there is a northern arc, which connects the vertices $r$ and $r+1$ on top row and a southern arc, which connects the vertices $i$ and $r$ on bottom row. The propagating lines are the same as the ones appearing in the diagram of $x$, with one difference: The propagating line which connects the vertex $i$ on bottom row with the vertex $i'$ on top row in the diagram of $x$, it connects now the vertex $r$ on bottom row with the vertex $i'$ on top row. Therefore, $c(d)=\{NSp_{j_0}, p_2,\,p_3, \, p_{j_0-1},\,p_{j_0+1}, \dots, p_k\}$.
	
	Similarly,  in the diagram of $d'$ there are two arcs, which are the flipped arcs of the diagram of $d$ and the propagating lines remaining from the diagram of $x$ with again one difference: the propagating line which connects the vertex $i$ on top row with the vertex $i'$ on bottom row in the diagram of $x$, it connects now the vertex $r$ on top row with the vertex $i'$ on the bottom row. Therefore, we have $c(d')=\{NSp_{j_0}, p_2,\,p_3, \, p_{j_0-1},\,p_{j_0+1}, \dots, p_k\}=c(d)$.
	\end{proof}

We consider now the system  \eqref{sys}. According to Proposition \ref{pr1}, the only possible diagrams $x$ with $b_x\not=0$ are of the form  $$
		
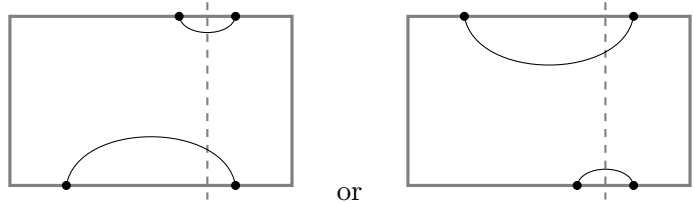
\begin{tikzpicture}[scale=0.75]
			\draw[gray,very thick] (0,0) rectangle (5,3);
			
			\foreach \x in {3,4} \filldraw (\x,3) circle (2pt);
			\foreach \x in {1,4} \filldraw (\x,0) circle (2pt);
			\draw[gray, thick, dashed] (3.5,3.25) to (3.5, -0.25);
			\draw (3,3) to [bend right=80] (4,3);
			\draw (1,0) to [bend left=80] (4,0);
				\end{tikzpicture}\,\,\,\,\,\,\, \text { or }\,\,\,\,\,\,\,	\begin{tikzpicture}[scale=0.75]
\draw[gray,very thick] (0,0) rectangle (5,3);

\foreach \x in {1,4} \filldraw (\x,3) circle (2pt);
\foreach \x in {3,4} \filldraw (\x,0) circle (2pt);
\draw[gray, thick, dashed] (3.5,3.25) to (3.5, -0.25);
\draw (1,3) to [bend right=80] (4,3);
\draw (3,0) to [bend left=80] (4,0);
\end{tikzpicture}$$
Note that the second case can be obtained from the first by applying the flip map $*$. Therefore, by Proposition \ref{pro}, it is enough to consider only the first case.

\begin{prop}
	Let $x, y\in B_{r,1}(\delta)$ be of the form $$
	
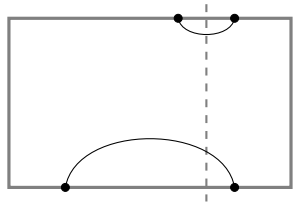
\begin{tikzpicture}[scale=0.75]
		\draw[gray,very thick] (0,0) rectangle (5,3);
		
		\foreach \x in {3,4} \filldraw (\x,3) circle (2pt);
		\foreach \x in {1,4} \filldraw (\x,0) circle (2pt);
		\draw[gray, thick, dashed] (3.5,3.25) to (3.5, -0.25);
		\draw (3,3) to [bend right=80] (4,3);
		\draw (1,0) to [bend left=80] (4,0);
	\end{tikzpicture}$$
with $y=\sigma x \sigma^{-1}$, for some $\sigma\in S_r$. Then, $b_x=b_y$.
\end{prop}
\begin{proof} We first notice that, since both diagrams $x$ and $y$ have  the northern arc that connects the vertices $r$ and $r+1$, we must have $\sigma \in S_{r-1}$. Therefore, $\sigma e=e \sigma$ and $\sigma^{-1} e=e \sigma^{-1}$.
	
	We now conjugate Equation \eqref{flip} by $\sigma$ and we get:
		$$\begin{array}{lcl}
		\sum\limits_{\mu}a_{\mu} \sum\limits_{d\in C_{\mu}}\sigma^{-1}(de-ed)\sigma&=&\sum\limits_{x\in \mathcal{B}}b_x\sigma^{-1}x \sigma\,\, \Longrightarrow \smallbreak\smallbreak\\
			\sum\limits_{\mu}a_{\mu} \sum\limits_{d\in C_{\mu}}(\sigma^{-1} de\sigma-\sigma^{-1} ed\sigma)&=&\sum\limits_{x\in \mathcal{B}}b_x\sigma^{-1} x \sigma\,\, \Longrightarrow \smallbreak\smallbreak\\
			\sum\limits_{\mu}a_{\mu} \sum\limits_{d\in C_{\mu}}(\sigma^{-1} d\sigma e-e\sigma^{-1} d\sigma)&=&\sum\limits_{x\in \mathcal{B}}b_x\sigma^{-1} x \sigma\,\, \stackrel{\ref{conj}}{\Longrightarrow} \smallbreak\smallbreak\\
				\sum\limits_{\mu}a_{\mu} \sum\limits_{d\in C_{\mu}}( de-e d)&=&\sum\limits_{x\in \mathcal{B}}b_x\sigma^{-1} x \sigma\,\, \Longrightarrow \smallbreak\smallbreak\\
				\sum\limits_{x\in \mathcal{B}}b_x x&=&\sum\limits_{x\in \mathcal{B}}b_{\sigma x \sigma^{-1}}x.
	\end{array}$$
		\end{proof}
	\begin{defn}\label{deff} 
	Let $x\in B_{r,1}(\delta)$ be of the form $$
		
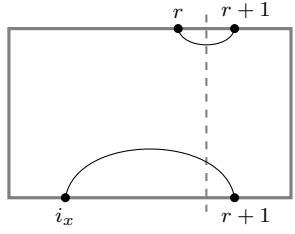
\begin{tikzpicture}[scale=0.75]
			\draw[gray,very thick] (0,0) rectangle (5,3);
			
			\foreach \x in {3,4} \filldraw (\x,3) circle (2pt);
			\foreach \x in {1,4} \filldraw (\x,0) circle (2pt);
			\draw[gray, thick, dashed] (3.5,3.25) to (3.5, -0.25);
			\draw (3,3) to [bend right=80] (4,3);
			\draw (1,0) to [bend left=80] (4,0);
				\draw (3,3) node[above]{\scriptsize $r$};
					\draw (4.2,3) node[above]{\scriptsize $r+1$};
						\draw (1,0) node[below]{\scriptsize $i_x$};
					\draw (4.2,0) node[below]{\scriptsize $r+1$};
		\end{tikzpicture}$$
	where $i_x\not=r$. The diagram $x$ defines the following bijection, given by the propagating lines in the diagram $x$:
	$$f_x: \{1,\dots, r-1\}\rightarrow \{1,\dots, \widehat{i_x}, \dots, r\}.$$

	We associate two diagrams in $B_{r,1}(\delta)$ to $x$ as follows:
	\begin{itemize} 
		\item[(i)] Define $\sigma_x \in S_r\times S_1\subseteq B_{r,1}(\delta)$ by
		$$\sigma_{x}(k)=\begin{cases}
			f_x(k), &\text{ if } k\in\{1,\dots, r-1\}\\
			i_x, &\text{ if } k=r\\
				r+1, &\text{ if } k=r+1
				\end{cases}$$
			\item[(ii)] Define $z_x \in B_{r,1}(\delta)$ as follows:
			The vertices $i_x$ and $r+1$ on top row (respectively, on bottom row) are connected by a northern arc (respectively, a southern arc). The propagating lines are given by the permutation $\tau_x:  \{1,\dots, \widehat{i_x}, \dots, r\}\rightarrow  \{1,\dots, \widehat{i_x}, \dots, r\}$, defined by
				$$\tau_{x}(k)=\begin{cases}
				f_x(k), &\text{ if } k\not=r\\
				f_x(i_x), &\text{ if } k=r
			\end{cases}$$
				\end{itemize} 
			\end{defn} 
	
			For example, let $x\in B_{6,1}(\delta)$ be the diagram $$
		
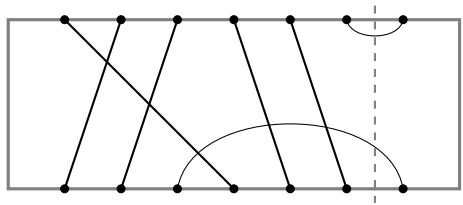
\begin{tikzpicture}[scale=0.75]
			\draw[gray,very thick] (0,0) rectangle (8,3);
			
			\foreach \x in {1,2,3,4,5,6,7} \filldraw (\x,3) circle (2pt);
			\foreach \x in {1,2,3,4,5,6,7} \filldraw (\x,0) circle (2pt);
			\draw[gray, thick, dashed] (6.5,3.25) to (6.5, -0.25);
			\draw (6,3) to [bend right=80] (7,3);
			\draw (3,0) to [bend left=80] (7,0);
			\draw[thick] (1,3) to (4,0);
			\draw[thick] (2,3) to (1,0);
			\draw[thick] (3,3) to (2,0);
			\draw[thick] (4,3) to (5,0);
			\draw[thick] (5,3) to (6,0);
		\end{tikzpicture}$$
		Then, we have: $f_x(1)=4,\, f_x(2)=1, \, f_x(3)=2,\, f_x(4)=5,\, f_{x}(5)=6$.
Therefore: 
$$\sigma_x=\vcenter{\hbox{\begin{tikzpicture}[scale=0.75]
	\draw[gray,very thick] (0,0) rectangle (8,3);
	
	\foreach \x in {1,2,3,4,5,6,7} \filldraw (\x,3) circle (2pt);
	\foreach \x in {1,2,3,4,5,6,7} \filldraw (\x,0) circle (2pt);
	\draw[gray, thick, dashed] (6.5,3.25) to (6.5, -0.25);
	\draw[thick] (6,3) to (3,0);
	\draw [thick] (7,3) to (7,0);
	\draw[thick] (1,3) to (4,0);
	\draw[thick] (2,3) to (1,0);
	\draw[thick] (3,3) to (2,0);
	\draw[thick] (4,3) to (5,0);
	\draw[thick] (5,3) to (6,0);
\end{tikzpicture}}},\,\,\,\,\,\,\,\,\,\,\,
z_x=\vcenter{\hbox{\begin{tikzpicture}[scale=0.75]
			\draw[gray,very thick] (0,0) rectangle (8,3);
			
			\foreach \x in {1,2,3,4,5,6,7} \filldraw (\x,3) circle (2pt);
			\foreach \x in {1,2,3,4,5,6,7} \filldraw (\x,0) circle (2pt);
			\draw[gray, thick, dashed] (6.5,3.25) to (6.5, -0.25);
			\draw[thick] (1,3) to (4,0);
			\draw[thick] (2,3) to (1,0);
			\draw[thick] (4,3) to (5,0);
			\draw[thick] (5,3) to (6,0);
				\draw[thick] (6,3) to (2,0);
					\draw (3,3) to [bend right=80] (7,3);
				\draw (3,0) to [bend left=80] (7,0);
		\end{tikzpicture}}}
$$
\begin{prop}\label{prop6}
Let $x$, $\sigma_x$ and $z_x$ be as in  \ref{deff}, then we have:
\begin{itemize}
	\item[(i)] $e\sigma_x=x$. Moreover, $c(\sigma_x)$ is obtained from $c(x)$ by replacing NS by L and $\sigma_x$ is the unique diagram $y\in B_{r,1}(\delta)$ with $c(y)$ without NS satisfying $ey=x$. 
		\item[(ii)] $ez_x=x$. Moreover, $c(z_x)$ is obtained from $c(x)$ by removing NS and adding it back as a trivial part and $z_x$ is the unique diagram $y\in B_{r,1}(\delta)$ with $c(y)$ having  NS only as trivial part satisfying $ey=x$. 
	\end{itemize}
\begin{proof}
	The result follows by definition of $\sigma_x$, $z_x$, concatenation of diagrams and the definition of walled generalized cycle type.
\end{proof}
\end{prop} 
\begin{cor}\label{cor}
	The set of equations $$\left\{b_x=0\,|\,
x=\vcenter{\hbox{	\begin{tikzpicture}[scale=0.5]
		\draw[gray,very thick] (0,0) rectangle (5,3);
		
		\foreach \x in {3,4} \filldraw (\x,3) circle (2pt);
		\foreach \x in {1,4} \filldraw (\x,0) circle (2pt);
		\draw[gray, thick, dashed] (3.5,3.25) to (3.5, -0.25);
		\draw (3,3) to [bend right=80] (4,3);
		\draw (1,0) to [bend left=80] (4,0);
		\draw (3,3) node[above]{\scriptsize $r$};
		\draw (4.2,3) node[above]{\scriptsize $r+1$};
		\draw (1,0) node[below]{\scriptsize $i_x$};
		\draw (4.2,0) node[below]{\scriptsize $r+1$};
	\end{tikzpicture}}}\,\,\,i_x\not=r, \text{ one } x \text { for each walled generalized cycle type } \right\}$$
viewed as equations in variables $a_{\mu}$ as defined in equation \eqref{eq} are linearly independent. 
\end{cor}
\begin{proof}
	We represent the aforementioned equations by a matrix, whose rows are the $b_x$'s and columns the coefficients $a_{\mu}$.
	
	By Proposition \ref{prop6}, each $b_x$ has a unique factor $a_{\mu}$ appearing with coefficient $1$, such that $\mu$ has no NS, namely $a_{c(\sigma_x)}$, and a unique factor $a_{\mu'}$ (appearing also with coefficient $1$), such that $\mu'$ has NS as a trivial part, namely $a_{c(z_x)}$.
	
	Now, note that $c(\sigma_x)$ and $c(z_x)$ determine $c(x)$,  therefore the matrix cannot have the following form:
		\[\begin{array}{@{} c @{}}
		\left [
		\begin{array}{ *{12}{c} }
			1 & 0 & \cdots&0& 	1 & 0 & \cdots&0& * & *& * & * \\
			1 & 0 & \cdots&0&	1 & 0 & \cdots&0& * & *& * & * \\
			\undermat{\text{no NS}}{&&\vdots&\phantom{\cdots}&}\undermat{\text{trivial NS}}{&&\vdots&\phantom{\cdots}&}\undermat{\text{non-trivial NS}}{
				&\vdots&\phantom{\cdots}&}\\
		\end{array}
		\right ] \\
		\mathstrut
	\end{array} 
	\]
	\\
	Therefore, the matrix is of the following form and, hence, the equations linearly independent.
	
		\[\begin{array}{@{} c @{}}
		\left [
		\begin{array}{ *{12}{c} }
			1 & 0 & \cdots&0& 	1 & 0 & \cdots&0& * & *& * & * \\
			1 & 0 & \cdots&0&	0 & 1 & \cdots&0& * & *& * & * \\
			0 & 1 & \cdots&0 & 	1 & 0 & \cdots&0& * & *& * & * \\
			0 & 1 & \cdots&0& 	0 & 1 & \cdots&0 &  * & *& * & * \\
			\undermat{\text{no NS}}{&&\vdots&\phantom{\cdots}&}\undermat{\text{trivial NS}}{&&\vdots&\phantom{\cdots}&}\undermat{\text{non-trivial NS}}{
			&\vdots&\phantom{\cdots}&}\\
		\end{array}
		\right ] \\
		\mathstrut
	\end{array} 
	\]
	
	\end{proof}

\begin{thm}\label{r,1}
	For any $\delta\in \mathbb{C}$, the centre of $B_{r,1}(\delta)$ is given by the algebra of supersymmetric polynomials evaluated at the Jucys-Murphy elements. Moreover, its dimension is given by $\dim Z(B_{r,1}(\delta)) = |\Lambda_{r,1}|$.
\end{thm}
\begin{proof} The result follows directly from Proposition \ref{gc} and  Corollary \ref{cor}.
\end{proof}
	
	\end{document}